\documentclass[3p,times]{elsarticle}

\usepackage{amsmath}
\usepackage{amsthm}

\usepackage{amssymb}
\usepackage[figuresright]{rotating}

\usepackage[font=footnotesize,labelfont=bf]{caption}
\usepackage[font=footnotesize,labelfont=bf]{subcaption}

\newtheorem{prop}{Proposition}
\newtheorem{cor}{Corollary}
\newtheorem{exmp}{Example}
\newtheorem{thm}{Theorem}

\newtheorem{rem}{Remark}

\newcommand{\sech}{ \sec\!\rm h}

\usepackage[english]{babel}
\usepackage{empheq,xcolor}

\begin{document}
	
	\begin{frontmatter}

		\title{Repeated derivatives of $\tanh$, $\sech$, $\dots $ and associated polynomials}

		\author[Enea]{G. Dattoli}
		\ead{giuseppe.dattoli@enea.it}
		
		\author[Enea]{S. Licciardi \corref{cor}}
		\ead{silvia.licciardi@enea.it}
		
		\author[Unict]{R.M Pidatella}
		\ead{rosa@dmi.unict.it}
		
		\author[Enea]{E. Sabia}
		\ead{elio.sabia@enea.it}
		
		\address[Enea]{ENEA - Frascati Research Center, Via Enrico Fermi 45, 00044, Frascati, Rome, Italy}
		\cortext[cor]{Corresponding author}
		\address[Unict]{Dep. of Mathematics and Computer Science, University of Catania, Viale A. Doria 6, 95125, Catania, Italy}

		\begin{abstract}
			Elementary problems like the evaluation of repeated derivatives of ordinary transcendent functions can usefully be treated by the use of special polynomials and of a formalism borrowed from combinatorial analysis. Motivated by previous researches in this field, we review the results obtained by other authors and develop a complementary point of view for the repeated derivatives of $\sec(.)$, $\tan(.)$ and for their hyperbolic counterparts.
		\end{abstract}

		
	\end{frontmatter}
	
	\section{Introduction}

The problem of finding closed forms for the repeated derivatives of trigonometric functions like tangent and secant, even though being an apparently elementary issue, has been solved in relatively recent times in ref. \cite{Adamchik}. Inspired by this work, a significant amount of research has been subsequently developed. In ref. \cite{Neto} the proof of the results of \cite{Adamchik} was reformulated in terms of a procedure exploiting the “Zeons” Algebra \cite{Feinsilver}. In \cite{Cvijovic,Boyadzhiev} authors addressed this study by employing a class of polynomials (the derivative polynomials ($DP$) introduced in refs. \cite{Hoffman,Hoffman2}) to reformulate the derivation and eventually get a set of fairly simple formulae, providing the successive derivatives of $\tan, \cot, \sec, \csc, \dots$ along with those of their hyperbolic counterparts. \\

In this note we develop a point of view not dissimilar from that of ref. \cite{Cvijovic,Boyadzhiev}. We provide straightforward results in terms of a single family of Legendre like polynomials and comment on the two forms of $DP$ introduced in \cite{Feinsilver,Cvijovic,Boyadzhiev}.\\ 

The repeated derivatives of composite functions $F(x)=f(g(x))$ is a well-established topic in calculus. The formulation of a procedure allowing the derivation of a formula comprising all the possible cases was established in the $XIX$ century \cite{Johnson} and opened important avenue of research in combinatorics \cite{Comtet} and umbral calculus \cite{S.Roman,SLicciardi} as well.\\

\noindent The problem is particularly interesting, encompasses different topics in analysis, including special polynomials like those belonging to the Touchard family \cite{DTouchard} and special numbers like the generalized Stirling forms \cite{Knuth,Mansour} of crucial importance in combinatorial analysis. The repeated derivatives of the Gaussian function are those of a composite function in which $f(.)$ is an exponential and $g(.)$ a quadratic function. The relevant expression leads to the Hermite polynomials as auxiliary tool, to get a synthetic expression for any order of the derivative \cite{Germano}. Within the same context, Bell polynomials emerge whenever one is interested to the derivatives of $F(x)=e^{g(x)}$ \cite{Comtet,Riordan,Bell}. The generalization to the case of $f(.)$, provided  by a generic infinitely  differentiable function, and $g(.)$, a quadratic form, has been discussed in refs. \cite{DSS,Bab} where the problem has been solved by the use of generalized nested forms of Hermite polynomials.\\

\noindent The hyperbolic secant is a composite function too in which $f(.)=(.)^{-1}$, $g(.)=\cosh (.)$. In refs. \cite{Neto,Cvijovic,Boyadzhiev,Nasa} interesting speculations have been presented on the study of the relevant repeated derivatives and of the associated auxiliary polynomials. In this paper we address the same problem, within a different context. Before entering the specific elements of the discussion we review the formalism we are going to exploit in this paper.\\

\noindent A pivotal role within the context of repeated derivatives is played by the Stirling number of second kind \cite{Knuth,Mansour}. They will be introduced using the Touchard polynomials which are defined through the Rodriguez type formula \cite{Germano,L.C.Andrews}.

\begin{equation}\label{TouchPol}
T_n(x)=e^{-x}(x\partial_x)^n e^x.
\end{equation}
They can be written in explicit form by the use of the following expression \cite{Knuth}

\begin{equation}\label{PartSt}
(x\partial_x)^n=\sum_{r=0}^n S_2 (n,r)x^r \partial_x^r
\end{equation}
with $S_2 (n,r)$ being Stirling number of second kind\footnote{We introduce the notation $S_2(l,m)$ instead of the usually symbol $\tiny\left\lbrace\! \begin{array}{c}
	l \\ m \end{array}\!\right\rbrace $.} \cite{Comtet,DTouchard}

\begin{equation}\label{key}
S_2 (l,m)=\frac{1}{m!}\sum_{j=0}^m (-1)^{m-j}\; \binom{m}{j}\;j^{\;l}.
\end{equation}
 According to eqs. \eqref{TouchPol}-\eqref{PartSt}, the Touchard polynomials are explicitly given by
 
 \begin{equation}\label{key}
 T_n(x)=\sum_{r=0}^n S_2(n,r)x^r 
 \end{equation}
and the numbers $S_2(n,k)$ are therefore the coefficients of the polynomials.\\

A fairly straightforward application of the previous computational tool is provided by the evaluation of a closed expression for higher order derivative.
 
\begin{thm}\label{thmSt}
	$\forall m\in\mathbb{N}$
	\begin{equation}\label{key}
\partial_x^m f(e^x)=\sum_{r=0}^m S_2 (m,r)\;e^{xr}f^{(r)}(e^x)
	\end{equation}
	where $f^{(r)}(\xi)$ denotes the $r$-order derivative . 
\end{thm}
\begin{proof}
	Let
	\begin{equation}\label{key}
	I_m(x)=\partial_x^m f(e^x),
	\end{equation}
by setting  $e^x=\xi$, we find 

\begin{equation}\label{impproof}
\partial_x^m f(e^x)=(\xi \partial_\xi)^m f(\xi)=\sum_{r=0}^m S_2 (m,r)\;\xi^r f^{(r)}(\xi)
\end{equation}
and in conclusion we obtain
\begin{equation*}\label{key}
\partial_x^m f(e^x)=\sum_{r=0}^m S_2 (m,r)\;e^{xr}f^{(r)}(e^x)
\end{equation*}	
in which it is understood that $f^{(r)}(e^x)=\partial_\xi^r f(\xi)\mid_{\xi=e^x}$.
\end{proof}

We define a family of auxiliary polynomials $P_n(x,y)$ and show how they can be very useful for the computation of the repeated derivatives of\footnote{We use the notation $\tan^{-1}(x)$ for $\arctan(x)$.} $\tan^{-1}(x)$.

\begin{thm}
Let 

\begin{equation}\label{LLP}
P_n (x,y)=n!\sum_{r=0}^{\lfloor\frac{n}{2}\rfloor}\frac{x^{n-2r}y^r (n-r)!}{(n-2r)!r!}, \quad \forall x,y\in\mathbb{R}, \forall n\in\mathbb{N},
\end{equation}
a family of two variable polynomials loosely ascribed to the Legendre \cite{DGM,DGM2} family,	then

	\begin{equation}\label{Dxntm}
	\partial_x^n\left(\tan^{-1}(x) \right) =\frac{1}{1+x^2}\;P_{n-1}\left(-\frac{2x}{1+x^2}, -\frac{1}{1+x^2} \right) .
	\end{equation}
\end{thm}
\begin{proof}	
We start from the computation of the quantity
\begin{equation}\label{key}
K_n (x)=\partial_x^n \left(\frac{1}{1+x^2} \right), \quad \forall x\in\mathbb{R}, \forall n\in\mathbb{N}
\end{equation}
which upon the use of the Laplace transform method can be written as
\begin{equation}\label{KnPn}
K_n (x)=\partial_x^n \int_0^\infty e^{-s(1+x^2)}ds= \int_0^\infty e^{-s}\partial_x^n e^{-sx^2}ds.
\end{equation}
The $n^{th}$ order derivative inside the integral sign can be explicetely worked in terms of two variable Hermite polynomials $H_n(x,y)$ \cite{Germano}, namely the auxiliary polynomials for the repeated derivatives of Gaussian functions according to the identity

\begin{equation}\label{key}
\partial_x^n \;e^{-ax^2}=H_n(-2ax,-a)\;e^{-ax^2}.
\end{equation}
Accordingly we find
\begin{equation}
\begin{split}
K_n (x)&=\int_0^\infty e^{-s} H_n(-2xs,-s)e^{-sx^2}ds =\int_0^\infty e^{-\sigma} H_n\left(-2\dfrac{x \sigma}{1+x^2}, -\dfrac{\sigma}{1+x^2} \right)\dfrac{1}{1+x^2}d\sigma=\\ 
& =\dfrac{n!}{1+x^2}\sum_{r=0}^{\lfloor\frac{n}{2}\rfloor}\dfrac{\left( \frac{-2x}{1+x^2}\right)^{n-2r} \left(\frac{-1}{1+x^2} \right)  }{(n-2r)!r!}\int_0^\infty e^{-\sigma}\sigma^{n-r}d\sigma =\frac{1}{1+x^2}\;P_n\left(-\frac{2x}{1+x^2}, -\frac{1}{1+x^2} \right) .
\end{split}
\end{equation}
%
The proof of the identity \eqref{Dxntm} is readily achieved by noting that
\begin{equation}\label{key}
\partial_x\left(\tan^{-1}(x) \right) =\left(\frac{1}{1+x^2}  \right),
\end{equation}
which eventually provides

\begin{equation*}\label{key}
\partial_x^n\left(\tan^{-1}(x) \right)  =K_{n-1}(x)=\frac{1}{1+x^2}\;P_{n-1}\left(-\frac{2x}{1+x^2}, -\frac{1}{1+x^2} \right) .
\end{equation*}
\end{proof}

\begin{cor}\label{CorPnnu}
An extension of the procedure we have just envisaged allows the further results

\begin{equation}\label{Pnnu}
\begin{split}
& K_n^\nu(x)=\partial_x^{n}\left(\frac{1}{\left( 1+x^2\right)^\nu}  \right) =\frac{1}{\left( 1+x^2\right)^\nu}\;P_n^\nu\left(-\frac{2x}{1+x^2}, -\frac{1}{1+x^2} \right),\\
& P_n^\nu(x,y)=\frac{n!}{\Gamma(\nu)}\sum_{r=0}^{\lfloor\frac{n}{2}\rfloor}\frac{x^{n-2r}y^r \;\Gamma(\nu+n-r)}{(n-2r)!r!}, \qquad \forall \nu\in\mathbb{R}.
\end{split}
\end{equation}
\end{cor}

\begin{exmp}
Corollary \ref{CorPnnu} can be exploited to write the repeated derivatives of $\cos^{-1}(x)$ in the form

\begin{equation}\label{key}
\partial_x^n \cos^{-1}(x) =-\partial_x^{n-1}\left( \frac{1}{\sqrt{1-x^2}}\right) =-\frac{1}{\sqrt{1-x^2}}P_{n-1}^{\frac{1}{2}}\left(\frac{2x}{1-x^2}, \frac{1}{1-x^2} \right).
\end{equation}
\end{exmp}

The previous examples have shown the interplay between special polynomials of the $P_n(x,y)$ and the derivatives of the inverse of trigonometric functions. \\

In the following we will see that the same polynomials are of central importance for the evaluation of the repetead derivatives of $\tan$ and $\sec$ functions. They play the role of auxiliary polynomials for the lorentzian type functions and inverse trigonometric functions as well. The relevant elements of contact with the $DP$ will be discussed in the final section of the paper.

\section{Higher Order Derivatives of Trigonometric Functions}

The starting point of the discussion of this section is the derivation of a closed form for the quantity $\partial_x^m\left(\sech(x) \right)$, $\forall m\in\mathbb{N}$.

\begin{prop}
	$\forall m\in\mathbb{N}, \forall x\in [0,2\pi]$,
	
\begin{equation}\label{key}
 \partial_x^m\left(\sech(x) \right) =\sech(x)\sum_{k=0}^m S_2(m,k)\;e^{(k-1)x}\left(e^x\;P_k\left( -\sech(x),-\dfrac{1}{2e^x} \sech(x)\right)  +k\;  P_{k-1}\left( -\sech(x),-\dfrac{1}{2e^x} \sech(x)\right)  \right).
\end{equation}
\end{prop}
\begin{proof}
	Let
\begin{equation}\label{key}
 \sech(x) =\dfrac{2}{e^x + e^{-x}} 
\end{equation}
then, after setting $e^x=\xi$, we obtain $\forall m\in\mathbb{N}$

\begin{equation}\label{key}
\partial_x^m\left(\sech(x) \right) =\partial_{\log \xi}^m \left(\dfrac{2\xi}{1+\xi^2} \right)= 
2\left(\xi \partial_\xi \right)^m \left( \dfrac{\xi}{1+\xi^2 }\right).  
\end{equation}
The straightforward applications of the proof of the Theorem \ref {thmSt} yields\footnote{We omit the argument of the $P_n$ polynomials for simplicity.} 

\begin{equation}\label{key}
2\left(\xi \partial_\xi \right)^m \left( \dfrac{\xi}{1+\xi^2 }\right)=\dfrac{2}{1+\xi^2}\sum_{k=0}^m S_2(m,k)\;\xi^k \left(\xi P_k+kP_{k-1} \right) 
\end{equation}
which, eventually, leads to

\begin{equation*}\label{key}
\partial_x^m\left(\sech(x) \right) =\sech(x)\sum_{k=0}^m S_2(m,k)\;e^{(k-1)x}\left(e^x\;P_k\left( -\sech(x),-\dfrac{1}{2e^x} \sech(x)\right)  +k\;  P_{k-1}\left( -\sech(x),-\dfrac{1}{2e^x} \sech(x)\right)  \right),
\end{equation*}
\end{proof}
We can go even further and obtain a closed form for $\partial_x^m\left(\sech(x) \right)^\nu$, $\forall m\in\mathbb{N}, \forall \nu\in\mathbb{R}$. 

\begin{cor}
$\forall m\in\mathbb{N}, \forall \nu\in\mathbb{R}$

\begin{equation}\label{key}
\partial_x^m\left(\sech^\nu(x) \right)=\sech^\nu(x)\sum_{r=0}^m S_2(m,r)\;e^{xr}\sum_{s=0}^r \binom{r}{s}\dfrac{n!}{(n-r+s)!}e^{-x(r-s)}\;P_s^\nu\left( -\sech(x),-\dfrac{1}{2e^x} \sech(x)\right)  
\end{equation}	
where $P_m^k(.\;,.)$ are the polynomials in eq. \eqref{Pnnu}.
\end{cor}

We can move on and provide the higher order derivatives for other trigonometric circular functions.

\begin{exmp}
$\forall m\in\mathbb{N}$, $\forall x\in [0,2\pi]$

\begin{equation}\label{key}
\partial_x^m\left(\sec (x) \right)=i^m \sec(x)\sum_{r=0}^m S_2(m,r)\;e^{ix(r-1)} \left(e^{ix}P_r \left(  -\sec (x),-\dfrac{1}{2e^{ix}} \sec(x)\right)+r \;P_{r-1}\left(  -\sec(x),-\dfrac{1}{2e^{ix}} \sec(x)\right) \right) .
\end{equation}	
\end{exmp}

\begin{exmp}
	$\forall m\in\mathbb{N}$, $\forall x\in [0,2\pi]:x\neq \frac{\pi}{2}+k\pi,\;k\in\mathbb{Z}$
	
	\begin{equation}\label{key}
	\partial_x^m\left(\tan (x) \right)=\sum_{s=0}^m \binom{m}{s} \sin\left(x+(m-s)\dfrac{\pi}{2} \right)\partial_x ^s(\sec (x)). 
	\end{equation}	
\end{exmp}

\begin{exmp}
	$\forall m\in\mathbb{N}$, $\forall x\in [0,2\pi]:x\neq k\pi,\;k\in\mathbb{Z}$
	
	\begin{equation}\label{key}
	\partial_x^m\left(\cot (x) \right)=\sum_{s=0}^m \binom{m}{s} \cos\left(x+(m-s)\dfrac{\pi}{2} \right)\partial_x ^s\left( \sec \left( x-\dfrac{\pi}{2}\right) \right) . 
	\end{equation}	
\end{exmp}

\begin{rem}
The use of the Leibniz rule can be avoided by setting 

\begin{equation}\label{key}
\tan(x)=i\;\dfrac{1-\xi^2 }{1+\xi^2}=i\;\left( \dfrac{2}{1+\xi^2 }-1\right), \quad \xi=e^{ix} 
\end{equation}
and eventually ending up with 

\begin{equation}\label{key}
	\partial_x^m\left(\tan (x) \right)=i^{m+1} \left( \sec(x) \sum_{r=0}^m S_2(m,r)\;e^{ix(r-1)} P_r\left( -\sec(x),-\dfrac{1}{2e^{ix}}\sec(x)\right) -\delta_{m,0}\right) 
\end{equation}
indeed 

\begin{equation*}\label{key}
\begin{split}
& \partial_x^m\left(\tan (x) \right)=
i^m \left(\xi\partial_\xi \right)^m  \left(i\;\left( \dfrac{2}{1+\xi^2 }-1\right) \right) =
2\; i^{m+1}	\left(\xi\partial_\xi \right)^m    \left( \dfrac{1}{1+\xi^2 }-\dfrac{1}{2}\right)  =2\; i^{m+1} \sum_{r=0}^m S_2(m,r)\;\xi^r  \left( \dfrac{1}{1+\xi^2 }-\dfrac{1}{2}\right)^{(r)}=\\
& = 2\; i^{m+1} \sum_{r=0}^m S_2(m,r)\;\xi^r \left(  \left(  \dfrac{1}{1+\xi^2 }\right)^{(r)}+\left(-\dfrac{1}{2} \right)^{(r)}\right)   =  2\; i^{m+1}\left(\left(\sum_{r=0}^m S_2(m,r)\;\xi^r  \dfrac{1}{1+\xi^2}P_r(.,.)\right)+ \left(\sum_{r=0}^m S_2(m,r)\;\xi^r \left(-\dfrac{1}{2} \right)^{(r)}\right) \right)  =\\
& =i^{m+1} \left( \sec(x) \sum_{r=0}^m S_2(m,r)\;e^{ix(r-1)} P_r\left( -\sec(x),-\dfrac{1}{2e^{ix}}\sec(x)\right) -\delta_{m,0}\right) .
\end{split}
\end{equation*}

\end{rem}

In the forthcoming section we will discuss the comparison with previous papers and present possible developments along the lines we have indicated.

\section{Final Comments}

In the previous two sections we have dealt with a general procedure useful to derive closed formulae for the repeated derivatives of circular and hyperbolic functions and of their inverse. We have underscored the importance of the polynomials $P_n(x,y)$ which play a role analogous to the $DP$ introduced in refs. \cite{Hoffman,Hoffman2} and used in \cite{Cvijovic,Boyadzhiev} for analogous occurrences.\\

The $P_n(x,y)$ are two variable polynomials defined in terms of the Laplace transform of the Hermite Kamp\'e d\'e F\'eri\'et family \cite{Appel}. They have been loosely defined Legendre-like and can be reduced to more familiar forms, by noting that

\begin{equation}\label{key}
\begin{split}
& P_n(x,y)=y^{\frac{n}{2}}P_n\left( \dfrac{x}{\sqrt{-y}}\right) ,\\
& P_n(z)=n!\sum_{r=0}^{\lfloor\frac{n}{2}\rfloor}\dfrac{(-1)^r z^{n-2r}(n-r)!}{(n-2r)!r!}
\end{split}
\end{equation}
with $P_n(z)$ satisfying the generating function 

\begin{equation}\label{key}
\sum_{n=0}^\infty \dfrac{t^n}{n!} P_n(z)=\dfrac{1}{1-tz+t^2}, \qquad \forall t,z\in\mathbb{R}:\mid t^2-t\;z \mid <1
 \end{equation}
and

\begin{equation}\label{key}
\frac{1}{n!}P_n(2x,-1)=\sum_{k=0}^{\lfloor\frac{n}{2}\rfloor} \dfrac{(-1)^k(n-k)!(2x)^{n-2k}}{(n-2k)!k!}=U_n(x).
\end{equation}

The link with previous papers addressing the same problem treated here, in particular with that developed in refs. \cite{Cvijovic,Boyadzhiev} can be obtained by using the same steps suggested by Cvijovi\'{c} or Boyadzhiev. We assume that an identity of the type 

\begin{equation}\label{dtan}
\partial_x^n (\tan(x))=\Pi_n (\tan(x))
\end{equation}
with $\prod_n(.)$ being not yet specified polynomials\footnote{In refs. \cite{Cvijovic,Boyadzhiev} they are denoted by $P_n(x)$, we have used the capital Greek letter to avoid confusion with the $P_n(x,y)$ polynomials defined in this paper.}, does holds.\\

The polynomials $\Pi_n(.)$ can be determined by the change of variable $\tan(x)=\xi$ which allows to transform eq. \eqref{dtan} into

\begin{equation}\label{derPi}
\left[\left( 1+\xi^2\right) \partial_\xi\right]^n(\xi)=\Pi_n(\xi) 
\end{equation}
which is a kind of Rodrigues type relation \cite{Rainville} defining the polynomials $\Pi_n(\xi)$.\\

\noindent The relevant generating function can be obtained by multiplying both sides of eq. \eqref{derPi} by $\dfrac{t^n}{n!}$, by summing up over the index $n$ and ending up with

\begin{equation}\label{key}
e^{t\left[\left( 1+\xi^2\right) \partial_\xi\right]}\xi=\sum_{n=0}^\infty \dfrac{t^n}{n!} \Pi_n(\xi).
\end{equation}
The use of the Lie derivative identity \cite{DOVT} 

\begin{equation}\label{etf}
e^{t\left[\left( 1+\xi^2\right) \partial_\xi\right]}f(\xi)=f\left(\dfrac{\xi\cos(t)+\sin(t)}{\cos(t)-\xi\sin(t)} \right) 
\end{equation}
finally yields

\begin{equation}\label{key}
\sum_{n=0}^\infty \dfrac{t^n}{n!} \Pi_n(\xi)=\dfrac{\xi+\tan(t)}{1-\xi\tan(t)}
\end{equation}
which is the generating function given in \cite{Cvijovic,Boyadzhiev}.\\

The second family of $DP$ can be defined by the use of the same procedure. According to refs. \cite{Cvijovic,Boyadzhiev} they are implicitly defined by the condition

\begin{equation}\label{key}
\partial_x^n (\sec(x))=\sec(x)Q_n(\tan(x))
\end{equation}
which, by the use of the same change of variable leading to eq. \eqref{derPi}, yields

\begin{equation}\label{key}
Q_n(\xi)=\dfrac{1}{\sqrt{1+\xi^2}}\left( (1+\xi^2)\;\partial_\xi\right)^n \sqrt{1+\xi^2}
\end{equation}
which can be straightforwardly exploited to derive the relevant properties. The use of the identity \eqref{etf} eventually yields the associated generating functions 

\begin{equation}\label{key}
\sum_{n=0}^\infty \dfrac{t^n}{n!}Q_n (\xi)=\dfrac{\sec(t)}{1-\xi\tan(t)}
\end{equation}
which is the same reported in refs. \cite{Cvijovic,Boyadzhiev}.\\


\noindent By recalling that the Hoppe formula writes 

\begin{equation}\label{key}
\begin{split}
& \partial_t^m f(g(t)) =\sum_{k=0}^m \dfrac{1}{k!} f^{(k)}(\sigma)\mid_{\sigma=g(t)}A_{m,k},\\
& A_{m,k}=\sum_{j=0}^k \binom{k}{j}(-1)^{k-j}g(t)^{k-j}\partial_t^m(g(t)^j), \quad \forall m\in\mathbb{N},
\end{split}
\end{equation}
 for the case of $\sec(x)$, we find 


\begin{equation}\label{key}
\partial_x^m \sec(x)=\sum_{k=0}^m  \sec^{k+1}(x)\sum_{j=0}^k \binom{k}{j}(-1)^{j}\cos^{k-j}(x)\;\partial_x^m(\cos^j(x))
\end{equation}

which appears impractical since it needs the repeated derivatives of an integer power of the cosine function.\\

Before closing the paper we consider worth to underscore the possible use of $DP$ to evaluate the derivatives of the type

\begin{equation}\label{key}
\lambda_{m,j}(x)=\partial_x^m (\cos^j(x)).
\end{equation}
Making the ansatz that 

\begin{equation}\label{key}
\partial_x^m (\cos^j(x))=(-1)^j\Lambda_{m,j}(\tan(x))
\end{equation}
and, by using the same procedure leading to eq. \eqref{derPi}, we can identify the function

\begin{equation}\label{key}
\Lambda_{m,j}(\xi)=\left[\left( 1+\xi^2\right)\partial_\xi  \right]^m \left( \dfrac{1}{\left( \sqrt{ 1+\xi^2}\right) ^{\;j} }\right)  
\end{equation}
which is specified by the generating function 

\begin{equation}\label{key}
\sum_{m=0}^\infty \dfrac{t^m}{m!}\Lambda_{m,j}(\xi)=
\dfrac{\left[ 1-\xi \tan(t)\right] ^j}{\left[\left(1+\xi^2 \right)\left(1+\tan^2(t) \right)   \right]^{\frac{j}{2}} }
\end{equation}

It is worth to consider the alternative assumption

\begin{equation}\label{key}
\partial_x^m (\cos^j(x)) =(-1)^m\Delta_{m,j}(\cos(x))
\end{equation}
which, after the change of variable $\xi=\cos(x)$, yields for the function $\Delta_{m,j}(.)$ the operational definition 
\begin{equation}\label{key}
\Delta_{m,j}(\xi)=\left(-\sqrt{1-\xi^2}\;\partial_\xi \right)^m \xi^j
\end{equation}
The use of the Lie derivative identity \cite{DOVT}

\begin{equation}\label{key}
e^{-t\;\left(\sqrt{1-\xi^2}\;\partial_\xi \right)}f(\xi)=f\left(\xi\cos(t)-\sqrt{1-\xi^2}\sin(t) \right) 
\end{equation}
yields, for $\Delta_{m,j}(\xi)$, the generating function

\begin{equation}\label{key}
\sum_{m=0}^\infty \dfrac{t^m}{m!}\Delta_{m,j}(\xi)=\left(\xi-\sqrt{1-\xi^2}\;\tan(t) \right)^j \cos^j(t). 
\end{equation} 

The two methods we have discussed in this paper, namely the procedure based on the $DP$ of refs. \cite{Cvijovic,Boyadzhiev} or on the support polynomials $P_n(x,y)$, are complementary. There are no prevailing reasons to prefer one or the other method. The formulae associated with the first are more synthetic but $\Pi_n(\xi),\;Q_n(\xi)$ are not given explicitly and should be evaluated recursively (which becomes cumbersome for larger order of the derivative). On the other side, the use of the other procedure leads to less appealing formulae in terms of polynomials which are, however, explicitly given.\\

In a forthcoming investigation we will discuss further extension of the method and we will go deeper into the links with previous researches.\\

\textbf{Acknowledgements}\\

\noindent The work of Dr. S. Licciardi was supported by an Enea Research Center individual fellowship.\\
Prof. R.M. Pidatella wants to thank the fund of University of Catania "Metodi gruppali e umbrali per modelli di diffusione e trasporto" included in "Piano della Ricerca 2016/2018 Linea di intervento 2'' for partial support of this work.\\

\textbf{Author Contributions}\\

\noindent Conceptualization: G.D.; methodology: G.D., S.L.; data curation: G:D., S.L.; validation: G.D., S.L., R.M.P.; formal analysis: G.D., S.L., R.M.P.; writing - original draft preparation: G.D., E.S.; writing - review and editing: S.L..\\

{}

\end{document}